\documentclass[preprint]{imsart}

\usepackage{a4wide,amsfonts,amsmath,amsthm,latexsym,amssymb,euscript,graphicx,units,mathrsfs,
hyperref, color, enumerate, comment}

\RequirePackage[colorlinks,citecolor=blue,urlcolor=blue]{hyperref}

\newtheorem{theorem}{Theorem}[section]
\newtheorem{example}[theorem]{Example}
\newtheorem{lemma}[theorem]{Lemma}
\newtheorem{proposition}[theorem]{Proposition}
\newtheorem{corollary}[theorem]{Corollary}
\newtheorem{definition}[theorem]{Definition}
\newtheorem{remark}[theorem]{Remark}

\startlocaldefs

\newcommand{\R}{\mathbb{R}}

\newcommand{\E}{\mathbb{E}}

\newcommand{\III}{\mathcal{I}}

\newcommand{\bea}{\begin{eqnarray}}
\newcommand{\ena}{\end{eqnarray}}
\newcommand{\beas}{\begin{eqnarray*}}
\newcommand{\enas}{\end{eqnarray*}} 
\numberwithin{equation}{section}

\endlocaldefs

\begin{document}

\begin{frontmatter}

  \title{Distances between nested densities and a measure of the
    impact of the prior in Bayesian statistics} 
\runtitle{Distances
    between nested densities and the impact of the prior}

  \author{\fnms{Christophe}
    \snm{Ley}\ead[label=e]{christophe.ley@ugent.be}\thanksref{t1},
    \fnms{Gesine}
    \snm{Reinert}\ead[label=e2]{reinert@stats.ox.ac.uk}\thanksref{t2}
    \and \fnms{Yvik}
    \snm{Swan}\corref{}\ead[label=e3]{yswan@ulg.ac.be}\thanksref{t3} }
  \affiliation{Ghent University,  University of Oxford \and
    Universit\'e de Li\`ege }

  \thankstext{t1}{Christophe Ley, whose affiliation during parts of this work was Universit\'e libre de Bruxelles, thanks the Fonds National de la
    Recherche Scientifique, Communaut\'e fran\c{c}aise de Belgique, for
    financial support via a Mandat de Charg\'e de Recherche FNRS.}
  \address{Christophe Ley \\   Ghent University\\  
   Department of Applied Mathematics, \\Computer Science and Statistics \\ Krijgslaan 281, S9\\ 
   9000 Ghent, Belgium \\
    \printead{e}} 

  \thankstext{t2}{Gesine Reinert acknowledges support from EPSRC grant
    EP/K032402/I as well as from the Keble College Advanced Studies
    Centre. }
  \address{Gesine Reinert \\ University of Oxford\\Department of Statistics\\1 South Parks Road\\Oxford OX1 3TG, UK\\
    \printead{e2}} 
 
  \thankstext{t3}{Yvik Swan gratefully acknowledges support from the
    IAP Research Network P7/06 of the Belgian State (Belgian Science
    Policy).}
  \address{Yvik Swan \\Universit\'e de Li\`ege\\ D\'epartement de Math\'ematique\\12 all\'ee de la d\'ecouverte\\ 
B\^at. B37 pkg 33a\\
4000 Li\`ege, Belgium \\
    \printead{e3}}


\runauthor{Ley, Reinert \and Swan}

\begin{abstract}
  In this paper we propose tight upper and lower bounds for the
  Wasserstein distance between any two {{univariate continuous
      distributions}} with probability densities $p_1$ and $p_2$
  having nested supports. These explicit bounds are expressed in terms
  of the derivative of the likelihood ratio $p_1/p_2$ as well as the
  Stein kernel $\tau_1$ of $p_1$.  The method of proof relies on a new
  variant of Stein's method which manipulates Stein operators.

  We give several applications of these bounds. Our main
  application is in Bayesian statistics : we derive explicit
  data-driven bounds on the Wasserstein distance between the posterior
  distribution based on a given prior and the no-prior posterior based
  uniquely on the sampling distribution. This is the first finite
  sample result confirming the well-known fact that with
  well-identified parameters and large sample sizes, reasonable
  choices of prior distributions will have only minor effects on
  posterior inferences if the data are benign.

\end{abstract}


\begin{keyword}
\kwd{Stein's method}
\kwd{Bayesian analysis}
\kwd{Prior distribution}
\kwd{Posterior distribution}
\end{keyword}

\end{frontmatter}


\section{Introduction}
\label{sec:introduction}
 
A key question in Bayesian analysis is the effect of the prior on the
posterior, and how this effect could be assessed. As more and more
data are collected, will the posterior distributions derived with
different priors be very similar? This question has a long history;
see for example (\cite{stein1965approximation, DF86b,DF86a}). While
asymptotic results which give conditions under which the effect of the
prior wanes as the sample size tends to infinity can be found for
example in \cite{DF86b,DF86a}, here we are interested, at \emph{fixed}
sample size, in explicit bounds on some measure of the distributional
distance between posteriors based on a given prior and the no-prior
data-only based posterior, allowing to detect at fixed sample size the
effect of the prior.

In the simple setting of prior and posterior being univariate and
continuous, the basic relation that the posterior is proportional to
the prior times the likelihood leads to the more general problem of
comparing two distributions $P_1$ and $P_2$ whose densities $p_1$ and
$p_2$ have nested supports. Letting $\mathcal{I}_1$ (resp.,
$\mathcal{I}_2$) be the support of $p_1$ (resp., $p_2$) and assuming
$\mathcal{I}_2 \subset \mathcal{I}_1$ we can write $$p_2 = \pi_0 p_1$$
for $\pi_0 = p_2/p_1$ a non-negative finite function called likelihood ratio in statistics.
To assess the distance between such distributions, we choose the
Wasserstein-1  distance 
defined as
\begin{equation}\label{dist}
 d_{\mathcal{W}}(P_1,P_2) 
 =\sup_{h\in\mathcal{H}}|\E[h(X_2)]-\E[h(X_1)]|
\end{equation}
for $\mathcal{H} = {\rm Lip}(1)$ the class of Lipschitz-1 functions,
where $X_1$ has distribution $P_1$ (resp.,
probability distribution function (pdf) $p_1$) and $X_2$ has distribution $P_2$
(resp., pdf $p_2$).
The central aim of this paper is to provide meaningful bounds on $d_{\mathcal{W}}(P_1, P_2)$ in
terms of  $\pi_0$.


Our approach to this problem relies on 
Stein's \emph{density approach} introduced in 
\cite{Stein1986, Stein2004}, as further developed in \cite{ley2014approximate, LS13b,LS13p,LS12a}. Let $P_1$ have density $p_1$
with interval support $\mathcal{I}_1$ with closure $[a_1, b_1]$ for
some $-\infty \le a_1 < b_1 \le +\infty$.  Suppose also that $P_1$ has mean 
$\mu$.  Then a notion which will be of particular importance
is the Stein kernel of  $P_1$ which is the function
$\tau_1: [a_1, b_1] \rightarrow \R$ given by
\begin{equation*}
\tau_1 (x) = \frac{1}{p_1(x)} \int_{a_1}^x ( \mu - y) p_1(y)dy.
\end{equation*}  
Our main results assume that $p_1$ and $p_2$ are absolutely continuous
densities, and that $\pi_0$ is a differentiable function satisfying

\

\noindent Assumption A
:~$ \lim_{x\rightarrow
  a_1}\pi_0(x)\int_{a_1}^x(h(y)-\E[h(X_1)])p_1(y)dy=0=\lim_{x\rightarrow
  b_1} \pi_0(x)\int_{x}^{b_1} (h(y)-\E[h(X_1)])p_1(y)dy $
for all Lipschitz-continuous functions $h$ with $|\E [h(X_1)]| < \infty$. Here $X_1 \sim P_1$. 

\

\noindent Under these assumptions we prove  the following result (Theorem \ref{maintheo}). 

\bigskip 
 {\bf{Theorem}}. {\it{ 
The Wasserstein distance between $P_1$ with pdf $ p_1$ and $P_2$ with pdf $p_2=\pi_0p_1$ satisfies the following  inequalities:
$$
  \left|  \E \left[
\pi_0'(X_1) \tau_1(X_1) \right] \right|  \le d_{\mathcal{W}} (P_1, P_2) \le   \E \left[
  \left| \pi_0'(X_1)\right|  \tau_1(X_1) \right]  
$$
where $\tau_1$ is the Stein kernel associated with $p_1$ and $X_1 \sim P_1$. 
}}

\bigskip If $P_1= \mathcal{N}(\mu, \sigma^2)$ is a normal distribution then the above
result simplifies considerably because $\tau_1(x) = \sigma^2$ is
constant, yielding
\begin{equation*}
  \sigma^2 \left| \E \left[ \pi_0'(X_1) \right] \right| \le
  d_{\mathcal{W}} (P_1, P_2) \le \sigma^2 \E \left[ \left|
      \pi_0'(X_1)\right| \right].
\end{equation*}
The Gaussian is characterized by the fact that its Stein kernel is
constant. More generally, all distributions belonging to the classical
Pearson family possess a polynomial Stein kernel (see
\cite{Stein1986}).  The problem of determining the Stein kernel is, in
general,  difficult. Even when the Stein kernel $\tau_1$ is not
available we can give the following simpler bound (Corollary~\ref{varcor}).

\bigskip
 {\bf{Corollary}}. {\it{ 
Under the same assumptions as in Theorem \ref{maintheo},
$$
\left| \E [X_1] - \E [X_2]\right| \le d_{\mathcal{W}} (P_1, P_2) \le  \| \pi_0'\|_{\infty} \mathrm{Var} [X_1]. 
$$
}}

More generally, because the Stein kernel is always positive, the upper
and lower bounds in the Theorem turn out to be the same whenever the
likelihood ratio $\pi_0$ is monotone, which is equivalent to requiring
that $P_1$ and $P_2$ are stochastically ordered in the sense of
likelihood ratios. This brings our next result (Corollary~\ref{theo2}). 

\bigskip {\bf{Corollary}}. {\it{ Let $X_1 \sim P_1$ and $X_2 \sim P_2$.  If $X_1\le_{LR}X_2$ or
    $X_2 \le_{LR}X_1$ then
\begin{align*}
  d_{\mathcal{W}}(P_1, P_2) & = |\E[X_2]- \E
                              [X_1]|  
                               = \E \left[  |\pi_0'(X_1) |
                              \tau_1(X_1) \right] 
                               = \E \left[ \left|\left( \log \pi_0(X_2) \right)'\right|
                              \tau_1(X_2) \right].
\end{align*} }}

\bigskip In case of a monotone likelihood ratio
between $P_1$ and $P_2$, the first of the above identities is easy to
derive directly from the known alternative definitions of the
Wasserstein distance (see e.g.~\cite{vallender1974calculation})
 \begin{align*}
   d_{\mathcal{W}}(P_1, P_2) = \int_{-\infty}^{\infty} \left| F_{P_1}(x) - F_{P_2}(x)  \right| dx
   = \int_0^1 \left| F_{P_1}^{-1}(u) - F_{P_2}^{-1}(u)  \right| du
 \end{align*}
 with $F_{P_1}$ and $F_{P_1}^{-1}$ (resp., $F_{P_2}$ and $F_{P_2}^{-1}$) the
 cumulative distribution function and quantile function of $P_1$
 (resp., $P_2$). 




We illustrate the effectiveness of our bounds in several  examples
at the end of Section~\ref{sec:result}, comparing e.g.~Gaussian random variables
or Azzalini's skew-symmetric densities with their symmetric
counterparts. 
In Section \ref{sec:Bayes} we treat as main application the Bayes
example wherein we {measure explicitly the effect of priors on
  posterior distributions}.  Suppose we observe data points
$x := (x_1, x_2, \ldots, x_n)$ with sampling density $f(x ; \theta)$
(proportional to the likelihood), where $\theta$ is the
one-dimensional parameter of interest.  {Let $p_0(\theta)$ be a
  certain prior distribution, possibly improper, and let $\Theta_2$ be
  the resulting posterior guess for $\theta$ perceived as a random
  variable. By Bayes' theorem, this has density
  $p_2(\theta;x)=\kappa_2(x) {f(x ; \theta)p_0(\theta)}$ with
  $\kappa_2(x)$ the normalizing constant which depends on the
  data. Under moderate assumptions, we provide computable expressions
  for the Wasserstein distance
  $ d_{\mathcal{W}}(\Theta_2,{\Theta}_1) $ between this posterior
  distribution and $\Theta_1$, whose law is the no-prior posterior
  distribution with density (proportional to the likelihood) given by
  $p_1(\theta; x) = \kappa_1(x) f(x;\theta)$, again with normalizing
  constant $\kappa_1(x)$ depending on the data.  The bounds we derive
  are expressed in terms of the data, the prior and the Stein kernel
  $\tau_1$ of the sampling distribution.}

We study the normal model with general and normal priors, the binomial
model under a general prior, a conjugate prior, and the Jeffreys' prior. We also consider the Poisson model with an
exponential prior, in which case we can make use of the likelihood
ratio ordering.
For example, with a normal
${\cal{N}} (\mu, \delta^2)$ prior and a random sample
$x_1, \ldots, x_n$ from a normal ${\cal{N}} (\theta, \sigma^2)$ model with fixed $\sigma^2$,
we obtain in \eqref{eq:21} that
\begin{equation*}
  \frac{\sigma^2}{n\delta^2 + \sigma^2}  \left|  \bar{x}- \mu \right|\le
  d_{\mathcal{W}} (\Theta_1, \Theta_2)
  \le  \frac{\sigma^2}{n\delta^2 +
    \sigma^2}  \left|  \bar{x}- \mu \right| + \frac{\sqrt{2}}{\sqrt{\pi}} \frac{\sigma^3}{n \delta
    \sqrt{\delta^2 n + \sigma^2}}.
\end{equation*}
Not only do we see that for $n\rightarrow\infty$, the distance becomes
zero, as is well known, but we also have an explicit dependence on the
difference between the sample mean $\bar{x}$ and the prior mean $\mu$, indicating the importance of a reasonable choice for the prior.
For a normal ${\cal{N}} (\theta, \sigma^2)$ model and a general prior on $\theta$, we obtain in \eqref{normalverygenprior}  that
\begin{eqnarray*} 
  \frac{\sigma^2}{n} |\E[\rho_0(\Theta_2)]|
  \leq d_\mathcal{W}(\Theta_1,\Theta_2)    \le \frac{\sigma^2}{n}\E[|\rho_0(\Theta_2)|]
\end{eqnarray*}
with $\rho_0$ the score function of the prior distribution.  Here the
data are hidden in the distribution of $\Theta_2$. 
In the binomial case with conjugate prior we obtain
\begin{align*}
& \frac{1}{n+2}\left|(2-\alpha-\beta)\frac{\frac{\alpha}{n}+\bar{x} }{1 + \frac{\alpha+\beta}{n}}+(\alpha-1)\right|
  \le d_{\mathcal{W}}(\Theta_1, \Theta_2)  \\
& \quad \qquad \qquad \qquad \le \frac{1}{n+2}\left(  |2-\beta-\alpha| \frac{\frac{\alpha}{n}+\bar{x} }{1 + \frac{\alpha+\beta}{n}}
  + |\alpha-1| \right),
\end{align*}
with $\alpha$ and $\beta$ the parameters of the conjugate (beta) prior. Finally in the Poisson case we obtain
\begin{eqnarray*}
  d_{\mathcal{W}}(\Theta_1, \Theta_2) &=& \frac{\lambda}{n + \lambda} \bar{x} + \frac{\lambda}{n(n+\lambda)}.
\end{eqnarray*}
with $\lambda>0$ the parameter of the exponential prior.

\medskip {{ The main tool in this paper is a specification of the
    general approach in \cite{ley2014approximate} which allows to
    manipulate Stein operators.  Distributions can be compared through
    their Stein operators which are far from being
    unique; for a single distribution there is a whole family of
    operators which could serve as Stein operators, see for example
    \cite{ley2014approximate}. In this paper, for probability
    distribution $P$ with pdf $p$ we choose the Stein operator
    $\mathcal{T}_P$ as
 $$
    \mathcal{T}_P :   f \mapsto \mathcal{T}_Pf = \frac{(fp)'}{p}
    $$
  with the convention that $ \mathcal{T}_Pf(x) = 0$ outside of the support of $P$; for details 
see Definition \ref{steinpair} and \cite{LS12a}.  For this choice of operator,   the product
  structure  implies a convenient connection between
  $\mathcal{T}_1$, the Stein operator for $P_1$ with pdf $p_1$, and $\mathcal{T}_2$, the Stein operator for $P_2$ with pdf $p_2 = \pi_0 p_1$, namely
\begin{equation*}
   \mathcal{T}_2(f)=\mathcal{T}_1(f)+ f\frac{\pi_0'}{\pi_0} =\mathcal{T}_1(f)+ f(\log \pi_0)' ;  
\end{equation*}
see \eqref{fund}. The difference
$$  \mathcal{T}_2(f)- \mathcal{T}_1(f)=  f(\log \pi_0)' $$
is the cornerstone of our results. }}

\begin{remark}
  This paper restricts attention to the univariate case. The
  multivariate case is of considerable interest but our approach
  requires an extension of the density method to a multivariate
  setting, which is to date still under construction and not yet
  available.
  
  {{
  Using the approach in \cite{ley2014approximate} it would be possible to extend our results to more general Radon-Nikodym derivatives, at the expense of clarity of exposition. }}
\end{remark}

The paper is organized as follows. In Section~\ref{sec:2}, we provide
the necessary notations and definitions from Stein's method, which
allows us to state our main result, Theorem~\ref{maintheo}, in
Section~\ref{sec:result}. Several applications of this result are
discussed in Examples~\ref{sec:dist-betw-prod-1} to \ref{sec:exo4}, while
Section~\ref{sec:Bayes} tackles our motivating Bayesian problem by
providing a measure of the impact of the choice of the prior on the
posterior distribution for finite sample size $n$. Finally in
Section~\ref{sec:stein-factors-1} we provide a proof of one of the
crucial bounds we need for our estimation purposes.

\section{A review of Stein's density approach}\label{sec:2}

\subsection{Notations and definitions}\label{sec:setting}

Here we recall some notions from \cite{ley2014approximate} and
\cite{LS12a}.  Consider a {{probability distribution $P$}}  with  continuous
univariate {{Lebesgue}} density $p$ on the real line and let $L^1(p) {{= L^1(p(x) dx)}}$ denote the
collection of $f : \R\to \R$ such that
$\E |f(X)| = \int |f(x)|p(x)dx<\infty$, where $X \sim P$. Let
$\mathcal{I} = \left\{ x \in \R \, | \, p(x)>0 \right\}$ be the
support of $p$. {{In this paper we shall use the following definition of a Stein operator; see for example \cite{ley2014approximate} for a discussion of alternative choices.}}

\begin{definition}\label{steinpair}[Stein pair]
  The Stein class $ \mathcal{F}(P)$ of $P$ is the collection of
  $f:\R\to \R$ such that (i) $fp$ is absolutely continuous, (ii)
  $(fp)'\in L^1(dx)$ and (iii) $\int_{\R}(fp)'dx = 0$. The Stein
  operator $\mathcal{T}_P$ for $P$ is
  \begin{equation}
    \label{eq:12}
    \mathcal{T}_P : \mathcal{F}(P) \to L^1(p) :   f \mapsto \mathcal{T}_Pf = \frac{(fp)'}{p}
  \end{equation}
  with the convention that $ \mathcal{T}_P f(x) = 0$ outside of
  $\mathcal{I}$.   
\end{definition}
{{Here $(fp)'$ denotes the derivative of $fp$ which exists Lebesgue-almost surely due to the assumption of absolute continuity. Often the Stein pair $(\mathcal{F}(P),\mathcal{T}_P)$ is written as dependent on $X\sim P$ rather than on $P$ (that is, as $(\mathcal{F}(X),\mathcal{T}_X)$); we use the dependence on the distribution to emphasize that the pair itself is not random.}}

Note that because we only consider $f$ multiplied by $p$ the behavior
of $f$ outside of $\mathcal{I}$ is irrelevant. 

\begin{remark} \label{constant} A sufficient condition
for $\mathcal{F}(P)\neq \emptyset$ is that $p'$ is integrable with
integral $0$ so that e.g.~$f=1 \in \mathcal{F}(P)$.  Such an
assumption is in general too strong (see e.g.~\cite{Stein2004} for a
discussion about the arcsine distribution) and weaker assumptions on
$p$ are permitted in our framework, although in such cases stronger
constraints on the functions in $\mathcal{F}(P)$ are necessary.  {{In particular the constant functions may not belong to  $\mathcal{F}(P)$.}} 

All
random quantities appearing in the sequel will be assumed to have
non-empty Stein class (an assumption verified for all classical
distributions from the literature).
\end{remark}

It is easy to see {{from Definition \ref{steinpair} (iii)}} that $\E[ \mathcal{T}_Pf(X)]= 0$ for all
$f \in \mathcal{F}(P)$. More generally one can prove that if $Y$ and
$X$ share the same support then $ Y\stackrel{\mathcal{D}}{=} X$ (equality in distribution) if
and only if $\E \left[ \mathcal{T}_Pf(Y) \right]=0$ for all
$f \in \mathcal{F}(P)$.  For any family of operators ${\mathcal{T}}$ indexed by univariate probability measures $P$ and $Q$ and for any class of functions $\mathcal{G}$  we say that $(\mathcal{T}_P, \mathcal{G}) $ is a 
\emph{Stein characterization} if 
\bea \label{steinpaireq}
 P = Q     \Longleftrightarrow \mathcal{T}_Q (f) =  \mathcal{T}_P (f) \quad \forall f \in \mathcal{G};
 \ena 
see \cite{LS12a,LS13b} for  general versions.  In particular  a Stein pair $(\mathcal{T}_P, \mathcal{F}(P))$   is a Stein characterization. 

With our notations, the operator $\mathcal{T}_P$ also admits an
inverse which is easy to write out formally at least. 
 Let $X \sim P$ have  (open, closed, or
 half-open) interval  support ${\cal{I}}$ between $a$ and $b$, where
 $-\infty \le a < b \le+ \infty$ and 
 $$\mathcal{F}^{(0)}(P)= \{ h \in L^1(p): \E[h(X)]=0 \}.$$ 
Define $\mathcal{T}_P^{-1} :
 \mathcal{F}^{(0)}(P) \to \mathcal{F}(P)$ by 
\begin{equation}
  \label{eq:1}
  \mathcal{T}_P^{-1}h(x) = \frac{1}{p(x)} \int_{a}^x h(y) p(y)dy=-\frac{1}{p(x)} \int_x^{b} h(y) p(y)dy.
\end{equation}  
The operator $\mathcal{T}_P^{-1}$ is the \emph{inverse Stein
  operator} of $P$  in the sense that 
  $$ \mathcal{T}_P ( \mathcal{T}_P^{-1}h) = h.$$  
Note how the particular structure of the r.h.s.~of \eqref{eq:1} ensures
that $\mathcal{T}_P^{-1}h$  belongs to $\mathcal{F}(P)$ for any $h \in
\mathcal{F}^{(0)}(P)$. 
If in addition 
$(fp)(a) = (fp)(b) = 0$ for all $f \in \mathcal{F}(p)$ then 
  $$ \mathcal{T}_P^{-1} ( \mathcal{T}_Pf) = f$$ 
so that $\mathcal{T}_P^{-1}$ constitutes a bona fide inverse in this
case.

\subsection{Standardizations of the operator}
\label{sec:stand-oper}
Although the Stein pair $(\mathcal{T}_P, \mathcal{F}(P))$ is unique{{ly defined in Definition \ref{steinpair},}} 
there are many implicit conditions on $f \in \mathcal{F}(P)$ which are
useful to identify before applying this construction to specific
approximation problems. {{In particular for favourable behavior of the inverse Stein operator it may be advantageous to consider}} only subclasses
$\mathcal{F}_{\mathrm{sub}}(P)\subset \mathcal{F}(P)$ of functions
satisfying certain target-specific and well chosen constraints. A good
choice of subclass will lead to specific forms of the resulting
operator which may turn out {{to have a smooth inverse Stein operator, as illustrated in the next example. As long as $\mathcal{F}_{\mathrm{sub}}(P)$ is a measure-determining class, the class is informative enough to satisfy \eqref{steinpaireq}. }}

\begin{example}
  In the case of the Laplace distribution $\mathrm{Lap}$ with pdf $p(x) \propto e^{-|x|}$  {{the Stein operator from Definition \ref{steinpair}}} is 
  \begin{equation}\label{eq:18}
    \mathcal{T}_{\mathrm{Lap}}f(x) = f'(x) - \mathrm{sign}(x)f(x)
  \end{equation}
  with $f \in \mathcal{F}(\mathrm{Lap})$, {{the class of}} functions such that $f(x) e^{-|x|}$
  is differentiable {{almost surely}} with integrable derivative{{, and the derivative of $f(x) e^{-|x|}$ integrates to 0 over the real line}}. This
  operator does not have agreeable properties, mainly because the
  assumptions on $\mathcal{F}(\mathrm{Lap})$ are not explicit (see
  e.g.~\cite{eichelsbacher2014malliavin} and \cite{pikeren}). 
  It is {{indeed}} sufficient to consider functions of the form
  $f(x) = (x f_0(x)e^{|x|})'/e^{|x|}$ for certain functions
  $f_0$. Applying $\mathcal{T}_{\mathrm{Lap}}$ to such functions
  yields the second order operator
\begin{equation}\label{eq:25}
  \mathcal{T}_{\mathrm{Lap}}f(x) = \mathcal{A}_Xf_0(x) =    xf_0''(x)+2f_0'(x)-x f_0(x)
\end{equation}
with $f_0 \in \mathcal{F}(\mathcal{A}_\mathrm{Lap})$ the class of 
functions {{which are piecewise twice continuously differentiable such that $x f_0''(x), f_0'(x)$ and $x f_0(x)$ are all in $L^1 \left( e^{-|x|} dx \right) $}}, as considered e.g.~in
\cite{eichelsbacher2014malliavin,gaunt2014variance}. {{ In \cite{pikeren} functions of the form $f(x) = (-(g(x) - g(0)) e^{|x|})'/e^{|x|}$  yielded the second order operator 
$$  \mathcal{T}_{\mathrm{Lap}, PR}g(x) = g(x) - g(0) -  g''(x)$$ 
for $g$ locally absolutely continuous with
$g \in L^1 \left( e^{-|x|} dx \right) $, $g'$ also locally absolutely
continuous and $g'' \in L^1 \left( e^{-|x|} dx \right) $.  The
operator $\mathcal{T}_{\mathrm{Lap}, PR}$ is also discussed in
\cite{eichelsbacher2014malliavin} but not used in
\cite{eichelsbacher2014malliavin} because it did not fit in with
Malliavin calculus as well as \eqref{eq:25}. }}
\end{example}

{{
Even in the straightforward situation of a normal distribution, often a standardization is applied, as explained in the next example.}}

\begin{example}
For the standard normal distribution 
${\mathcal{N}}(0,1)$ it is
easy to write out the  operator \eqref{eq:12} explicitly to get
$ \mathcal{T}_{{\mathcal{N}}(0,1)} (f)(x) = f'(x) - xf(x)$ acting on a wide class of
functions $\mathcal{F}({\mathcal{N}}(0,1))$  which includes all absolutely continuous functions with
polynomial decay at $\pm \infty$. In particular the constant function $\bf{1}$ is in $\mathcal{F}({\mathcal{N}}(0,1))$. {{A standardization of the form 
$f(x)= H_n(x) f_0(x)$ with $H_n$ the $n^{th}$ Hermite polynomial ($H_0(x) = 1, H_1(x) = x, H_2(x) = x^2 -1$) gives as  operator 
$ {\cal{A}}f_0 (x) = H_n(x) f_0'(x) - H_{n+1}(x) f_0(x),$
see for example \cite{goldrein}. 

It is also possible to study the behavior of functions $f_h$ under
quite general conditions on $h$. For instance if $\mathcal{H}$ is the set
of measurable functions $h: \R \to [0,1]$ (leading to the total
variation measure) then $\mathcal{F}^{(1)}$ is contained in the
collection of {{differentiable}} functions such that $\| f \| \le \sqrt{\pi/2}$ and
$\|f'\| \le 2$; see for instance~\cite{NP11}.

For the general normal distribution 
${\mathcal{N}}(\mu,\sigma^2)$ 
the  operator \eqref{eq:12} gives 
\bea \label{normalop}
 \mathcal{T}_{{\mathcal{N}}(\mu,\sigma^2)} (f)(x) = f'(x) - \frac{x-\mu}{\sigma^2} f(x). 
 \ena 
 The standardization $f(x) = \sigma^2 g'( x)$  yields the classical Ornstein-Uhlenbeck Stein operator 
$
 {\mathcal{A}} g (x) = \sigma^2 g''(x) - (x - \mu) g'(x),
$
 see for example~\cite{ChGoSh11}.}}
\end{example}

We call the passage from a parsimonious operator $\mathcal{T}_P$ (such
as \eqref{eq:18}) acting on the implicit class $\mathcal{F}(P)$ to a
specific operator $\mathcal{A}_P$ (such as \eqref{eq:25}) acting on a
generic class $\mathcal{F}(\mathcal{A}_P)$ a \emph{standardization} of
$(\mathcal{T}_P, \mathcal{F}(P))$. Given $P$ there are 
infinitely many different possible standardizations. 
\subsection{The Stein transfer principle}
\label{sec:steins-transf-princ}

Suppose that we aim to assess the discrepancy between the laws of two
random quantities $X$ {{with distribution $P$}} and $W$ {{with distribution $Q$}}, say, in terms of some probability
distance of the form
\begin{equation}\label{eq:2}
d_{\mathcal{H}}(P, Q)  = 
d_{\mathcal{H}}(X, W) =   \sup_{h\in \mathcal{H}} | \E[ h(W)] - \E [h(X)]|, 
\end{equation}
for $\mathcal{H}$ some measure{{-determining}} class; many common distances
can be written under the form \eqref{eq:2}, including the Kolmogorov
distance (with $\mathcal{H}$ the collection of indicators of half
lines), the Total Variation distance (with $\mathcal{H}$ the
collection of indicators of Borel sets) and the $1$-Wasserstein
distance (see \eqref{dist}). {{Here writing $d_{\mathcal{H}}(X, W) $ is a shorthand for \eqref{eq:2}: this distance is not random.}} 

Let $P$ have Stein pair $(\mathcal{T}_P, \mathcal{F}(P))$ and consider
a standardization $(\mathcal{A}_P, \mathcal{F}(\mathcal{A}_P))$ as
described in Section \ref{sec:stand-oper}. The first key idea in
Stein's method is to relate the test functions $h$ of interest to a
function $f = f_h \in \mathcal{F}(\mathcal{A}_P)$ through the
so-called {\it{Stein equation}}
\begin{equation} \label{Steinequation} 
h(x) - \E [h(X)] = \mathcal{A}_P f(x) , \quad x \in \mathcal{I},
\end{equation}
so that, for $f_h$ solving \eqref{Steinequation}, we get
$h(W) - \E [h(X)] = \mathcal{A}_P f_h(W)$ and, in particular,
\begin{equation} \label{eq:start} 
\sup_{h \in \mathcal{H}} | \E [ h(W)]- \E [h(X)]  | =
 \sup_{f \in  \mathcal{F}^{(1)}} | \E \left[ \mathcal{A}_P f(W) \right] |
\end{equation}
where
$ \mathcal{F}^{(1)} =\mathcal{F}^{(1)} (\mathcal{A}_P, \mathcal{H})= \left\{ f \in \mathcal{F}(\mathcal{A}_P) \, | \,
  \mathcal{A}_Pf = h  - \E [h(X)]  \mbox{ for some } h \in \mathcal{H} \right\}. $
The first step in Stein's method thus consists in some form of
transfer principle whereby one transforms the problem of bounding the
distance $d_{\mathcal{H}}(P, Q)$ into that of bounding the
expectations of the operators $\mathcal{A}_P$ over a specific class of
functions.

\begin{example}
For the standard normal distribution, the operators \eqref{eq:12} and \eqref{normalop} give 
$ \mathcal{T}_{{\mathcal{N}}(0,1)} (f)(x) = f'(x) - xf(x)$. 
 Bounding expressions of the form $\left|\E \left[ f'(W) - Wf(W) \right]\right|$ as occurring in the  r.h.s. of \eqref{eq:start} is a potent starting
point for Gaussian approximation problems. Prominent examples include 
 $W = \sum_i \xi_i$ a standardized sum of weakly dependent variables, and $W = F(X)$  a
functional of a Gaussian process; see e.g.~\cite{ChGoSh11,Ro11,NP11} for an overview. 
\end{example}

In general, the success of Stein's method for a particular target relies
on the positive combination of three factors : 
\begin{enumerate}[(i)]
\item \label{item:1}   the functions in
$\mathcal{F}^{(1)}$ need to have ``good'' properties (e.g.~be bounded
with bounded derivatives),
\item \label{item:2} the operator
$\mathcal{A}_P$ needs to be amenable to computations (e.g.~its
expression should only involve polynomial functions), 
\item \label{item:3} there must be some ``handle'' on the expressions
  $\E \left[ \mathcal{A}_Pf(W) \right]$ (e.g.~{{allowing for Taylor-type expansions or the application of couplings)}}.
\end{enumerate}
Conditions \eqref{item:1} to \eqref{item:3} are satisfied 
for a great variety of target distributions (including
the exponential, chi-squared, gamma, semi-circle, variance gamma and
many others, see {{for example}} 
https://sites.google.com/site/yvikswan/about-stein-s-method for an
up-to-date list). 

\subsection{The Stein kernel}
\label{sec:score-function-stein}
One of the many keys to a successful application of Stein's method for
a given target distribution {{$P$}} lies in the properties of $P$'s 
\emph{Stein kernel}. 
We now review some properties of 
this 
quantity which will play a central role in our analysis; {{see \cite{ LS12a} or \cite{ley2014approximate}
for details.}}  

\begin{definition}
Let $P$ be a probability distribution with mean $\mu$, {{and let $X \sim P$}}.  A  \emph{Stein kernel} of $P$ is a random variable $\tau_P(X)$
  such that 
  \begin{equation}
    \label{eq:steqdef}
    \E \left[ \tau_P(X) \varphi'(X) \right] = \E \left[ (X-\mu) \varphi(X) \right]
  \end{equation}
for all {{differentiable}} $\varphi: \R \to \R$ for which the expectation {{$ \E \left[ (X-\mu) \varphi(X) \right]$}}
exists.
\end{definition}

 The function 
$  x \mapsto \tau_P(x) = \E \left[ \tau_P(X) \, | \, X = x
\right]
$
 is a \emph{Stein kernel (function)} of $P$.  If $P$ has interval
 support with closure $[a, b]$ then, letting $Id$ denote
 the identity function, it is not hard to see
 that 
 \begin{equation*}
   \tau_P(x)= \mathcal{T}_P^{-1}(\mu-Id)(x) = \frac{1}{p(x)} \int_a^x(\mu -y)p(y)dy
 \end{equation*}
 is the unique Stein kernel of $P$. Moreover the following properties of the
 Stein kernel are immediate consequences of its definition:
 \bea\label{sec:score-function-stein-1}
\mbox{for all } x \in \R  \mbox{ we have that } \tau_P(x)\geq 0
 \mbox{ and }\E \left[ \tau_P(X)  \right] = \mathrm{Var}[X]. 
    \ena
    The Stein kernels  for a wide variety of
    classical distributions (all members of the Pearson family, as it
    turns out) bear agreeable expressions; see~\cite[Table 1]{EdVi12},
    \cite{nourdin2013entropy,nourdin2013integration} or the
    forthcoming \cite{DRS14} for illustrations.

\subsection{Stein factors}
\label{sec:stein-factors}

Let $P$ have a continuous density $p$  with mean $\mu$ and support
$\mathcal{I}$ such that the closure of $\mathcal{I}$ is the interval
$[a,b]$ (possibly with infinite endpoints).  Let
$(\mathcal{T}_P, \mathcal{F}(P))$ be the Stein pair of $P$ and suppose
that $P$ admits a Stein kernel $\tau_P(x)$, as described in Subsection \ref{sec:score-function-stein}. We introduce the standardized Stein pair
$(\mathcal{A}_P, \mathcal{F}(\mathcal{A}_P))$ with
\begin{equation}
  \label{eq:29}
  \mathcal{A}_Pf(x) = \mathcal{T}_P(\tau_P f)(x) = \tau_P (x) f'(x) +
  (\mu - x) f(x), \quad x \in \mathcal{I},
\end{equation}
and 
\begin{align*}
  \mathcal{F}(\mathcal{A}_P) &  = \left\{f : \R \to \R \mbox{ absolutely continuous
                               such that
                               }  \right. \\
                             & \qquad \left.  \lim_{x\to a}   f(x)
                               \int_a^{x}(\mu-u) p(u) du =  \lim_{x\to b}   f(x)
                               \int_x^{b}(\mu-u) p(u) du = 0
                               \right.\\
& \qquad \left.\mbox{and }  \left( f(x)
                               \int_a^{x}(\mu-u) p(u) du \right)' \in L^1(dx)\right\}
\end{align*}
Our next lemma shows that whenever applicable, standardization
\eqref{eq:29} satisfies    requirement \eqref{item:1} 
from the end of Section \ref{sec:steins-transf-princ}. 

\begin{lemma}\label{sec:about-constants-1}
 
  Let $\mathcal{H} = \mathrm{Lip}(1)$ be the collection of Lipschitz
  functions $h : \R \to \R$ with Lipschitz constant 1 and let
  $\mathcal{F}^{(1)}$ be the collection of
  $f \in \mathcal{F}(\mathcal{A}_P)$ such that
  $\mathcal{A}_Pf = h - \E [h(X)]$ for some $h \in \mathcal{H}$. Then
  $\mathcal{F}^{(1)}$ is contained in the collection of functions $f$
  such that $\| f \|_{\infty} \le 1$.
\end{lemma}
Lemma~\ref{sec:about-constants-1} is strongly related to 
\cite[Corollary 2.16]{Do12}, adapted to our framework.  For the sake of
completeness we present a proof of (a generalization of) this result
at the end of the present paper. The key to our approach  lies in the
fact that the bound in Lemma~\ref{sec:about-constants-1} \emph{does
  not depend} on the standardization of the target $P$; it is in
particular independent of the mean and variance of $X \sim P$ or of any
normalizing constant that might appear in the expression of the
density of $P$. 


\section{Comparing univariate
  continuous densities}\label{sec:result}

For $i=1,2,$ let $P_i$ be a probability distribution with   an absolutely continuous 
density $p_i(\cdot)$ having   support  
  $\III_i$ with closure $\bar \III_i = [a_i, b_i]$,  for some
$-\infty \le a_i < b_i \le +\infty$. Suppose that
$\III_2\subset\III_1$ and define $\pi_0$ through 
   \begin{equation}
     \label{eq:27}
     p_2 = \pi_0 p_1.
   \end{equation}
   Associate with both {{distributions}}  the Stein pairs
   $(\mathcal{T}_i, \mathcal{F}_i)$ for $i=1, 2$, as well as the
   resulting construction from the previous section. 
   
   The product
  structure \eqref{eq:27} implies a {{key}} connection between
  $\mathcal{T}_1$ and $\mathcal{T}_2$, namely
\begin{equation}\label{fund}
   \mathcal{T}_2(f)=\mathcal{T}_1(f)+ f\frac{\pi_0'}{\pi_0} =\mathcal{T}_1(f)+ f(\log \pi_0)'  
\end{equation}
for all $f\in\mathcal{F}_1\cap\mathcal{F}_2$. 

\subsection{Bounds on the Wasserstein distance between univariate
  continuous densities}\label{sec:result}

   Our {{main}} 
   objective in this section is to provide computable and meaningful bounds on the Wasserstein distance
   $ d_{\mathcal{W}}(P_1, P_2)${{, defined in \eqref{dist},}} in terms of $\pi_0$ {{and $P_1$}}, under the product structure \eqref{eq:27}.

\begin{theorem}\label{maintheo}
 {{For $i=1,2,$ let $P_i$ be a probability distribution with   an absolutely continuous 
density $p_i$ having   support  
  $\III_i$ with closure $\bar \III_i = [a_i, b_i]$,  for some
$-\infty \le a_i < b_i \le +\infty$; suppose that
$\III_2\subset\III_1$  and let {{$X_i \sim P_i$ have finite means $\mu_i$ for $i=1, 2$}}. Assume that $\pi_0 = \frac{p_2}{p_1}$, defined on $\III_2$, is
  differentiable on $\III_2$, satisfies  
  $ \E |  (X_1 - \mu_1) \pi_0( X_1) | < \infty$ and  }} 
\begin{align}
  \label{eq:6}
&  \left(  \pi_0(x) \int_{a_1}^x(h(y)-\E[h(X_1)])p_1(y)dy \right)' \in
  L^1(dx)\\
& \lim_{x \to a_2, b_2} \pi_0(x) \int_{a_1}^x(h(y)-\E[h(X_1)])p_1(y)dy
  = 0 \label{eq:9}
\end{align}
for all $h \in \mathcal{H}$, the set of Lipschitz-1 functions on $\R$.  Then 
\begin{equation}
\label{eq:10}
\left|\E \left[ \pi_0'(X_1) \tau_1(X_1) \right]\right| \le d_{\mathcal{W}} (P_1, P_2) \le  
\E \left[   \left|  \pi_0'(X_1) \right| \tau_1(X_1) \right]
\end{equation}
where $\tau_1$ is the Stein kernel of $P_1$.
\end{theorem}
\begin{proof}
  We first prove the lower bound. {{Let $X_2 \sim P_2$.}} Start by noting that
  $d_{\mathcal{W}} (P_1, P_2)  \ge |\E [X_2] - \E [X_1]|$ because
  $Id\in{\rm Lip}(1)$.  {{With \eqref{eq:27} we get}} that 
  \begin{align}
    \E[X_2] - \E[X_1] & = \E [X_1\pi_0(X_1) ] - \mu_1 \nonumber\\
    & = \E \left[ (X_1-\mu_1)\pi_0(X_1)  \right]\nonumber\\
    & = \E \left[ \tau_1(X_1) \pi_0'(X_1) \right] \label{eq:43}
  \end{align}
  where we used the fact that $\E \left[ \pi_0(X_1) \right] = 1$ and the definition \eqref{eq:steqdef} of
  $\tau_1(X_1)$ in the last line.

  Next we prove the upper bound.  By the definition \eqref{eq:1},
  $ f_h=\mathcal{T}_1^{-1}(h-\E[h(X_1)]) \in \mathcal{F}_1$. On the
  other hand, Conditions \eqref{eq:6} and \eqref{eq:9} guarantee that
  $f_h \in \mathcal{F}_2$ for all $h$ because
  \begin{equation*}
    p_2 f_h =  \pi_0(x) \int_{a_1}^x(h(y)-\E[h(X_1)])p_1(y)dy 
  \end{equation*}
 is necessarily absolutely continuous. 
We conclude that all functions
$f_h=\mathcal{T}_1^{-1}(h-\E[h(X_1)])$ belong to the intersection
$\mathcal{F}_1\cap\mathcal{F}_2$. Hence
\begin{eqnarray}
\E[h(X_2)]-\E[h(X_1)]&=&\E[\mathcal{T}_1(f_h)(X_2)]\nonumber\\
&=&\E[\mathcal{T}_1(f_h)(X_2)]-\E[\mathcal{T}_2(f_h)(X_2)]\label{1}\\
&=&-\E[f_h(X_2)(\log \pi_0)'(X_2)].\nonumber
\end{eqnarray}
Equality~\eqref{1} follows from the
assumption that $f_h \in \mathcal{F}_2$ so that $\mathcal{T}_2f_h$
cancels when integrated with respect to $p_2$, whereas the last equality follows from Equation \eqref{fund}. Now we define
$g_h=f_h/\tau_1$ and recall that $\tau_1 \ge 0$ to get
$$
\left|\E[h(X_2)]-\E[h(X_1)]\right|=\left|\E\left[g_h(X_2)(\log
    \pi_0)'(X_2)\tau_1(X_2)\right]\right|\leq
||g_h||_\infty\E\left[\left|(\log
    \pi_0)'(X_2)\right|\tau_1(X_2)\right].
$$
It   follows from Lemma~\ref{sec:about-constants-1} that
$||g_h||_\infty\leq1$ for all $h\in{\rm Lip}(1)$, yielding  
\begin{align*}
  d_{\mathcal{W}}(P_1, P_2) & \le \E\left[\left|(\log
    \pi_0)'(X_2)\right|\tau_1(X_2)\right] = \E\left[\left|\pi_0'(X_1)\right|\tau_1(X_1)\right],
\end{align*}
the last equality again following from \eqref{eq:27}.
\end{proof}

Assumptions \eqref{eq:6} and \eqref{eq:9} are crucial. While
\eqref{eq:9} is in a sense innocuous (because $\mathcal{I}_2 \subset
\mathcal{I}_1$), \eqref{eq:6}  is quite stringent  yet hard to
verify in practice.
In Section \ref{sec:stein-factors-1} we provide a proof of the
following explicit and easy to verify sufficient conditions on $p$ for
these and hence Theorem \ref{maintheo} to hold.

\begin{proposition} \label{prop:easycond} We use the notations of
  Theorem \ref{maintheo}. Suppose that $\pi_0$, $p_1$ and $p_2$ are
  differentiable over their support and that their derivatives are
  integrable. Suppose that 
  \begin{equation*}
    \lim_{x \to a_2, b_2} \pi_0(x) p_1(x) \tau_1(x)  =     \lim_{x \to a_2, b_2} p_2(x) \tau_1(x)  = 0.
  \end{equation*}
Let $\rho_1 = p_1'/p_1$ and suppose also that 
\begin{equation*}
  \pi_0' p_1 \tau_1  =  p_2'  \tau_{1}  - \rho_1 \tau_1 p_2 \in L^1(dx).
\end{equation*}
  Then  Theorem \ref{maintheo} applies.
\end{proposition}

\begin{example}[Distance between
  Gaussians]\label{sec:dist-betw-prod-1} {To compare two Gaussian
    distributions, ${\cal{N}}(\mu_1, \sigma_1^2)$  and
    ${\cal{N}}(\mu_2, \sigma_2^2)$, order them so that $\sigma_2^2 \le
    \sigma_1^2$, and if $\sigma_1 = \sigma_2$ then assume that $\mu_1
    > \mu_2$.  
  If $P_1$ is ${\cal{N}}(\mu_1, \sigma_1^2)$ then
  $\tau_1(x) = \sigma_1^2$ is constant (see
  e.g.~\cite{Stein1986}).  With  $P_2$ being
  ${\cal{N}}(\mu_2, \sigma_2^2)$, all conditions in Proposition
  \ref{prop:easycond} are satisfied.   
Applying Theorem \ref{maintheo} and noting that $ (\log \pi_0(x))' = x
\left( \frac{1}{\sigma_1^2} - 
  \frac{1}{\sigma_2^2} \right) + \left( \frac{\mu_2}{\sigma_2^2} -
  \frac{\mu_1}{\sigma_1^2} \right)$, we obtain that \beas
  \label{eq:26}
  | \mu_2 - \mu_1 |\le d_{\mathcal{W}}(P_1, P_2) &\le &  {{\sigma_1^2 \E \left| X_2 \left(  \frac{1}{\sigma_1^2} -
    \frac{1}{\sigma_2^2} \right) +   \left( \frac{\mu_2}{\sigma_2^2} -
    \frac{\mu_1}{\sigma_1^2} \right) \right|  }} \nonumber \\
    &\le&  \left|
    \frac{\sigma_1^2}{\sigma_2^2} \mu_2- \mu_1  \right| +  \left(
    \frac{\sigma_1^2}{\sigma_2^2} -1  \right) \E \left|X_2 \right|. 
\enas
In the special case $\mu_2 = \mu_1 = 0$  we compute $\E \left|X_2
\right| = \sqrt{2/\pi} \sigma_2$ to get 
\begin{equation*}
  d_{\mathcal{W}}(P_1, P_2) \le \sqrt{\frac{2}{\pi}} \frac{\sigma_1^2-\sigma_2^2}{\sigma_2},
\end{equation*}
to be compared with  a similar result in \cite[Proposition
3.6.1]{NP11}. 

If $\mu_2 \neq 0$ then the general expression for
$ \E \left|X_2 \right|$ is not agreeable, which is why we suggest
using the inequality
$\E|X_2|\leq \left(\E[X_2^2]\right)^{1/2}=\sqrt{\sigma_2^2+\mu_2^2}$,
leading to
$$
| \mu_2 - \mu_1 |\le d_{\mathcal{W}}(P_1, P_2) \le \left|
    \frac{\sigma_1^2}{\sigma_2^2} \mu_2- \mu_1  \right| +  \left(
 \frac{\sigma_1^2}{\sigma_2^2} -1 \right) \sqrt{\sigma_2^2+\mu_2^2}.
$$
With $\mu_1=\mu_2=\mu$, the upper bound becomes $(|\mu|+\sqrt{\sigma_2^2+\mu^2})\left( 
    \frac{\sigma_1^2}{\sigma_2^2} -1 \right)$.  We have not found a
  similar result  in the literature  (outside of the centered case) and 
   {{computing the
 Wasserstein distance directly using \eqref{wasscalc} is prohibitive
 as the cdf's are not available in closed form.}}  }
\end{example}

\begin{remark}
  Our upper bounds are not restricted to the Wasserstein case only. Indeed,
  mimicking large parts of the proof of Theorem~\ref{maintheo}, we
  obtain the general bound
\begin{equation}
  \label{eq:20}
  d_{\mathcal{H}} (P_1, P_2) \le     \kappa_{\mathcal{H}} {{\E\left[\left|\pi_0'(X_1)\right|\tau_1(X_1)\right]}}  
\end{equation}
with
$\kappa_{\mathcal{H}}=\sup_{h\in\mathcal{H}}||\mathcal{T}_1^{-1}(h-\E_1h)/\tau_1||_\infty$
and $\mathcal{H}$ a measure-determining class of functions (the
Kolmogorov distance corresponds to the class of indicators of
half-lines, the Total Variation distance to the indicators of Borel
sets). Usefulness of \eqref{eq:20} hinges around availability of
bounds similar to Lemma~\ref{sec:about-constants-1} on the more
general constant $\kappa_{\mathcal{H}}$.
\end{remark}
Unravelling the lower bound and using~\eqref{sec:score-function-stein-1} in the upper bound of \eqref{eq:10} we also obtain the
following weaker but perhaps  more transparent result. 

\begin{corollary}\label{varcor}
  Under the same assumptions as {{for Theorem \ref{maintheo}, with $X_2 \sim P_2$, }} 
\begin{equation}
\label{eq:var}
\left| \E [X_2]-\E [X_1] \right|\le  d_{\mathcal{W}} (P_1, P_2) \le \| \pi_0'\|_{\infty} \mathrm{Var} [X_1]. 
\end{equation}
\end{corollary}

We shall use Corollary \eqref{eq:var} in Section \ref{sec:Bayes}. We
stress the fact that there is \emph{no} normalizing constant appearing
in the bounds \eqref{eq:10} and \eqref{eq:var}. Also, the absence of 
Stein kernel in \eqref{eq:var} is in somes cases an advantage
because the Stein kernel is not always easy to compute. 


There are many ways of expressing the Wasserstein distance
\eqref{dist} between two random variables. In general, if $P_1$ has cumulative distribution function (cdf) $F_{P_1}$ and if $P_2$ has
 cdf $F_{P_2}$  then
\begin{equation} \label{wasscalc} d_{\mathcal{W}}(P_1, P_2) = \int_\R |
  F_{P_1}(x) - F_{P_2}(x) | dx = \int_0^1 | F_{P_1}^{-1}(u) - F_{P_2}^{-1} (u) | du = \inf \E \left| \xi_1-\xi_2 \right|
\end{equation} 
where the infimum in this last expression is taken over all possible
couplings $(\xi_1, \xi_2)$ of $(P_1, P_2)$ (see
e.g.~\cite{vallender1974calculation, villani2008optimal}). Often exact
computable expressions of Wasserstein distances tend to be difficult
to obtain. The similarity between the upper and lower bounds in
\eqref{eq:10} encourages us to formulate the next result.
\begin{corollary}\label{theo2}
If $X_i \sim P_i, i=1, 2,$ {{are as in  Theorem \ref{maintheo}}} and if $\pi_0$ is monotone increasing or decreasing,  then 
\begin{equation}
  \label{eq:15}
  d_{\mathcal{W}}(P_1, P_2) = \left| \E[X_2]- \E [X_1] \right|
  = \E \left[  \left|\pi_0'(X_1)  \right|\tau_1(X_1) \right] {{= \E\left[\left|(\log
    \pi_0)'(X_2)\right|\tau_1(X_2)\right]}} .
\end{equation}
 \end{corollary} 
 Note how the second expression in \eqref{eq:15} can be immediately
 obtained from the first by applying the same argument as in
 \eqref{eq:43}. Now while the second expression in \eqref{eq:15} is
 new, the first is in fact not. Indeed the condition that $\pi_0$ be
 monotone in Corollary~\ref{theo2} is equivalent to requiring
 $X_1 \ge_{LR} X_2$ (stochastically ordered in the sense of likelihood
 ratio, see e.g.~\cite[Section 9.4]{ross1996stochastic} or
 Example~\ref{sec:exo1}).   If $ X_1 \le_{LR} X_2$
 then $F_{P_2} \le F_{P_1}$ (see for example \cite[Theorem
 1.C.4]{shaked2007stochastic}), so that
 $ d_{\mathcal{W}}(X_1, X_2) = \int_\R ( F_{P_1}(x) - F_{P_2}(x) ) dx = \E
 [X_2]-\E[X_1]$.

\begin{example}[{Distance between Azzalini-type skew-symmetric
  distributions}\label{sec:exo3}]
  Consider a symmetric density $p_1$ on the real line.  The so-called
\emph{Azzalini-type skew-symmetric distributions} are constructed from {{such a pdf}} 
$p_1$ by considering the densities $p_2(x)= 2p_1(x) G(\lambda x)$
with $G$ the {{cdf}} of a univariate symmetric distribution {{with pdf}}  $g$ and
$\lambda\in\R$ a parameter (called skewness parameter); see
\cite{hallin2014skew} for an overview of these skewing mechanisms and of their
applications.  The founding example is Azzalini \cite{Azza85}'s
\emph{skew-normal density} $2\phi(x)\Phi(\lambda x)$ (denoted
$\mathcal{SN}(0, 1, \lambda)$), where $\phi$ and $\Phi$ respectively
stand for the standard normal density and cumulative distribution
function.


Corollary~\ref{theo2} provides, under mild conditions on $g$ and $G$,
an exact expression for the Wasserstein distance between $P_1$ with
pdf $p_1$ and its skew-symmetric counterpart $P_2$ with pdf $p_2$
 since
in this case $(\log\pi_0)'(x)=\lambda g(\lambda x)/G(\lambda x)$ which
is positive or negative depending on the sign of $\lambda$ {{as both $g$ and $G$ are positive on the support of $P_2$}}. Thus we
have {{$\pi_0'(x) = 2 \lambda g(\lambda x)$ and }} 
\begin{equation*}
  d_{\mathcal{W}}(p_1, p_2)
 =2 |\lambda| \E
 \left[\tau_1(X_1)g(\lambda X_1) \right]. 
\end{equation*}
Perhaps the most interesting instance of the above is the comparison
of the {{standard}} normal with the skew-normal {{(all conditions in
    Proposition~\ref{prop:easycond} are satisfied in this case) :}} 
$$
d_{\mathcal{W}}\left(\mathcal{N}(0,1),\mathcal{SN}(0,1,\lambda)\right)=\sqrt{\frac{2}{\pi}}
\frac{|\lambda|}{\sqrt{1+\lambda^2}} 
$$
(recall that  $\tau_1(x) = 1$). Letting $\lambda\rightarrow\infty$ we obtain that the distance between
the half-normal with density $2\phi(x)\mathcal{I}_{x\geq0}$ and the
normal is $\sqrt{{2}/{\pi}}$, see also \cite{dobler2013stein}. As in
the previous example, such results are not easy to obtain directly
from \eqref{wasscalc}.
\end{example}

{{Likelihood ratio orderings have a natural role in comparing parametric densities.}} 
Let $p(x; \theta)$ be a parametric family of densities with parameter
of interest $\theta \in \R$ (see e.g. \cite{LS13p} for discussion
and references). Set $p_1(\cdot) = p(\cdot; \theta_1)$ and
$p_2(\cdot) = p(\cdot; \theta_2)$. The family $p(x; \theta)$ is said
to have monotone \emph{likelihood ratio} if
$x \mapsto p(x; \theta_2)/p(x; \theta_1)$ is non decreasing as soon as
$\theta_2>\theta_1$ (and vice-versa).  If $P_1$ has pdf $p_1$ and
if $P_2$ has pdf $p_2$ then under monotone likelihood ratio,
$P_2 \le  P_1$. The property of monotone likelihood ratio is
intrinsically linked with the validity of one-sided tests in
statistics, see \cite{KR1956}.

\begin{example}[{Distances within the exponential family}\label{sec:exo1}]

A 
noteworthy class of parametric distributions which satisfy the property of monotone likelihood ratio is the \emph{canonical regular exponential family}
$     p(x; \theta) = \ell(x) e^{\theta x - A(\theta)}$
for some scalar functions $\ell$ and $A$, with the range of the
distribution being independent of $\theta$, see for example
\cite[page 639]{KR1956}. If $\theta_1 > \theta_2$ then 
  $(\log\pi_0)'(x)=\left( \log \frac{p_2(x)}{p_1(x)} \right)' = \theta_2-\theta_1<0$
  for all $x \in \R$ and thus from
  \eqref{eq:15} we find with $X_2 \sim P_2$ that $ d_{\mathcal{W}}(P_1, P_2)  =  |\theta_2-\theta_1| \E \left[
    \tau_1(X_2) \right]$
under mild and easy-to-check conditions on $P_1$ and $P_2$. 
\end{example}

\begin{example}[{Distances between ``tilted'' distributions} \label{sec:exo4}]
 Fix a density $p_1$ with mean $\mu_1$ and consider, among all other
 densities $g$ with same support and fixed but different mean
 $\mu_2\neq \mu_1$, the density that minimizes the Kullback-Leibler divergence
\begin{equation*}
  KL(g|| p_1) = \int g(x) \log \left( \frac{g(x)}{p_1(x)} \right)dx. 
\end{equation*}
The Euler-Lagrange equation for the constrained variational problem is
$ \log g(x) = \log p_1(x) + \lambda_1x + \lambda_2$ solved by
\begin{equation}\label{eq:17}
  p_2(x) = p_1(x) \frac{e^{\lambda_1x}}{M_1(\lambda_1)} 
\end{equation}
with $M_1(t) = \E [e^{tX_1}]$ the moment generating function of $X_1\sim p_1$
and  $\lambda_1$ a solution to 
  \begin{align*}
    \frac{d}{dt}(\log M_1(t))_{t=\lambda_1} = \mu_2
  \end{align*}
  in order to guarantee $\E[X_2] = \mu_2$. We call \eqref{eq:17} a
  ``tilted'' version of $p_1$ (following the classical notion of
  exponential tilting, see e.g.~\cite{efron1981nonparametric}).  It is easy to compute
\begin{align*}
  KL(p_2 \, || \, p_1) = \lambda_1\mu_2 - \log M_1(\lambda_1).
\end{align*} 
Setting $\pi_0(x) = {e^{\lambda_1x}}/{M_1(\lambda_1)}$ we have
$\log(\pi_0)'(x) = \lambda_1$ and 
\begin{equation}
  \label{eq:19}
  d_{\mathcal{W}}(p_1, p_2) = |\lambda_1| \E \left[ \tau_1(X_2)
  \right]
\end{equation}
provided that the appropriate conditions are satisfied.

For the sake of illustration, take $p_1$ the Gamma distribution on the
positive half line with density
$p_1(x;\lambda,k) = \frac{1}{\Gamma(k)}
e^{-x/\lambda}x^{k-1}\lambda^{-k} $.
Then $M_1(t)=(1-\lambda t)^{-k}$ {{for $t < \frac{1}{\lambda}$}} and
$\lambda_1 = \frac{1}{\lambda}- \frac{k}{\mu_2}$. Moreover
$ \tau_1(x) = \lambda x$.  It is thus easy to check in this case that
all conditions in Proposition~\ref{prop:easycond} are satisfied. This
allows us to deduce from \eqref{eq:19} that
$$
d_{\mathcal{W}}(p_1, p_2) =  |\mu_2 - \lambda k|  
$$
which nicely complements $KL(p_2||p_1) = \frac{\mu_2}{\lambda}-k + \log \left(
  \frac{k\lambda}{\mu_2} \right)^k$ {{as an alternative comparison statistic}}.
\end{example}

 \section{On the  influence of the prior  in Bayesian statistics}
 \label{sec:Bayes}

 We now tackle the problem that motivated Theorem~\ref{maintheo} : assessing the impact of the choice of the prior
 distribution on the resulting posterior distribution in Bayesian
 statistics. In all examples we consider the conditions in
 Proposition~\ref{prop:easycond} are easy to verify explicitly.

We first fix the notations. {{Assume that the observation $x$ comes from a parametric model with pdf $f(x; \theta)$ with $\theta \in \Theta$ - $f(x; \theta)$ is often called the {\it{likelihood}} or the {\it{sampling density}}. We turn this model into a pdf for $\theta$ through 
 $$p_1(\theta; x) = \kappa_1(x) f(x; \theta)$$
 where $\kappa_1(x) = \left( \int f(x; \theta) d\theta \right)^{-1}$, and we assume that $\kappa_1 < \infty.$
 Let $P_1$ have pdf $p_1$ and call its Stein kernel $\tau_1$. Choose a possibly improper prior density $\pi_0(\theta)$, and let 
 $$ p_2(\theta; x) = \pi_0(\theta; x) p_1(\theta; x)$$ where 
 $$ \pi_0(\theta; x) = \kappa_2(x) \pi_0(\theta) \mbox{ such that } \int  p_2(\theta; x) d\theta =1.$$
 Then 
 $$ 1 = \int  p_2(\theta; x) d\theta  = \kappa_2(x) \int \pi_0(\theta) p_1(\theta; x) d\theta  = \kappa_2(x) \E[ \pi_0(\Theta_1)],$$
 where $\Theta_1$ has distribution $P_1$
 }} 
   which gives  an
  expression for the normalizing constant. {{Let $P_2 = P_2(\cdot;x)$ be the probability distribution on $\Theta$ with  pdf $p_2(\cdot; x) $. Then $P_2$ is the posterior distribution of $\theta$ under the prior $\pi_0$ and the data $x$;
  moreover $P_1$ can be seen as the distribution of $\theta$ under  a uniform prior and the data $x$.}}

Now we extract from
\eqref{eq:10} of Theorem~\ref{maintheo} the first bounds on the impact of a prior on the posterior distribution : 
  \begin{equation}
    \label{eq:24}
  \frac{\left|\E  \left[\tau_1(\Theta_1) \pi_0'(\Theta_1)
      \right]\right|}{\E [\pi_0(\Theta_1)]}\leq
    d_{\mathcal{W}}(P_2,P_1)  \leq \frac{\E
    \left[\tau_1(\Theta_1) \left|\pi_0'(\Theta_1) \right| \right]}{\E [\pi_0(\Theta_1)]}
  \end{equation}
which  can also be rewritten as
\begin{equation}
  \label{eq:22}
 \left| \E \left[ \Theta_2 \right] - \E \left[ \Theta_1 \right]
 \right| =   \left|\E  \left[\tau_1(\Theta_2) \rho_0(\Theta_2) \right]\right|\leq
    d_{\mathcal{W}}(P_2,P_1)  \leq \E
    \left[\tau_1(\Theta_2) \left|\rho_0(\Theta_2) \right| \right]
\end{equation}
with $\Theta_2 \sim P_2$ and 
\begin{equation*}
  \rho_0(\theta)  = \frac{\pi_0'(\theta)}{\pi_0(\theta)},
\end{equation*}
the score function {{of $\pi_0(\theta; x)$ with respect to $\theta$, which does not depend on the data $x$.}}
  As we shall see
in the forthcoming sections {{which treat some classical examples in Bayesian statistics}}, (\ref{eq:22}) often turns out to
be handier for computations than (\ref{eq:24}).


 \subsection{A normal model} \label{sec:Bayes2}

 Consider the simple setting where $x=(x_1, \ldots, x_n)$ is a random
 sample from a ${\cal N}(\theta, \sigma^2)$ population, where the
 scale $\sigma$ is known and the location $\theta$ is the parameter of
 interest, and assume that the prior $\pi_0 (\theta) > 0 $ for all
 $\theta \in \Theta$ is differentiable. The likelihood $f (x; \theta)$ of the normal model can be
 factorized into
\begin{align*}
  f (x; \theta) &= (2 \pi \sigma^2)^{-\frac{n}{2}}
                     \exp\left\{- \frac{1}{2} \sum_{i=1}^n \frac{(x_i - \theta)^2}{\sigma^2}
                     \right\}\\
                   &= (2 \pi \sigma^2)^{-\frac{n}{2}}
                     \exp\left\{-\frac{1}{2\sigma^2}\left(\sum_{i=1}^nx_i^2-n\bar{x}^2\right)\right\}\exp\left\{-
                     \frac{1}{2}
                     \frac{(\theta-\bar{x})^2}{\sigma^2/n}\right\} \\
                     &\propto {{ \exp\left\{-
                     \frac{1}{2}
                     \frac{(\theta-\bar{x})^2}{\sigma^2/n}\right\}}} \mbox{ when viewed as a function of $\theta$}
\end{align*}
where $\bar{x}=\frac{1}{n}\sum_{i=1}^nx_i$. Thus, 
{{$P_1 =  \mathcal{N}(\bar{x},\sigma^2/n)$.}} 
 Since  $\tau_1$ is
constant, equal to $\sigma^2/n$, the variance of $\Theta_1 \sim P_1$, the bound
\eqref{eq:24} becomes
\begin{align*}
  \frac{\sigma^2}{n} \frac{\left| \E \left[ \pi_0'(\Theta_1) \right]  \right|}{\E \left[ \pi_0(\Theta_1) \right]} \le
  d_{\mathcal{W}}(P_2, {{P_1}} )  \leq
  \frac{\sigma^2}{n}\frac{\E \left[\left|  \pi_0'(\Theta_1) \right|\right]}{\E \left[ \pi_0(\Theta_1) \right]} 
\end{align*}
and \eqref{eq:22} becomes
\begin{align} \label{normalverygenprior}
|\E[\Theta_2] - \bar{x}| =
  \frac{\sigma^2}{n} |\E[\rho_0(\Theta_2)]|
  \leq d_\mathcal{W}(P_1, P_2)    \le \frac{\sigma^2}{n}\E[|\rho_0(\Theta_2)|].
\end{align}
Both inequalities are equalities in the case that $\pi_0$ is monotone. 

\subsection{Normal prior and normal model}\label{exBayes2}

Consider the same setting as in the previous section with the
additional information that the prior $\pi_0$ is the density of a
${\cal N}(\mu, \delta^2)$, where $\mu$ and $ \delta^2>0$ are
known. Then the posterior $P_2$ is also normal, since
\beas p_2(\theta; {x}) &\propto& \exp\left\{- \frac{1}{2}
  \left(\frac{(\theta-\bar{x})^2}{\sigma^2/n}+ \frac{(\theta -
      \mu)^2}{\delta^2} \right) \right\}.  \enas Defining
$a= \frac{n}{\sigma^2} +\frac{1}{\delta^2}$ and
$b(x)= \frac{\bar{x}}{\sigma^2/n} + \frac{\mu}{\delta^2}$, we see that 
$P_2 = {\cal N}\left( \frac{b(x)}{a}, \frac{1}{a}\right)$.

Since the prior $\pi_0$ is not monotone, we cannot exactly evaluate the
Wasserstein distance between $P_1$ and $P_2$. However then
we can write $\rho_0(\theta) = - ({{\theta}}-\mu)/\delta^2$ to obtain 
\begin{equation}
  \label{eq:21}
 \frac{\sigma^2}{n\delta^2 + \sigma^2}  \left|  \bar{x}- \mu \right|\le
  d_{\mathcal{W}} (P_1, P_2)
\le \frac{\sigma^2}{n\delta^2 +
  \sigma^2}  \left|  \bar{x}- \mu \right| +  \frac{\sqrt{2}}{\sqrt{\pi}} \frac{\sigma^3}{n \delta
  \sqrt{\delta^2 n + \sigma^2}}.
\end{equation}
{To see this, the lower bound follows directly from simplifying the difference of the expectations, 
$$ \left|  \frac{b(x)}{a} - \bar{x}\right| = \frac{\sigma^2}{n\delta^2 + \sigma^2}  \left|  \bar{x}- \mu \right|.$$ For the upper bound, {{using
$\rho_0(\theta) = - ({{\theta}}-\mu)/\delta^2$ in \eqref{normalverygenprior}}} gives 
\beas
d_{\mathcal{W}} (P_1, P_2) &\le & \frac{\sigma^2}{n}\E\left[ |\rho_0(\Theta_2) |   \right] \\
&=& \frac{\sigma^2}{n\delta^2} \E [\left| \Theta_2 - \mu \right|] \\
&\le& \frac{\sigma^2}{n\delta^2} \left( \E\left[ \left| \Theta_2  - \frac{b(x)}{a} \right|\right] + \left|  \frac{b(x)}{a} - \mu \right| \right)  \\
&=& \frac{\sqrt{2}}{\sqrt{\pi}} \sqrt{\frac{1}{a}}
\frac{\sigma^2}{n\delta^2} + \frac{\sigma^2}{n\delta^2} \left|
  \frac{b(x)}{a} - \mu \right|
\\
&=& \frac{\sqrt{2}}{\sqrt{\pi}} \frac{\delta\sigma}{\sqrt{\delta^2 n +
    \sigma^2}} \frac{\sigma^2}{n\delta^2} + \frac{\sigma^2}{n\delta^2}
\frac{\delta^2}{\delta^2 + \frac{\sigma^2}{n}} \left| \bar{x}- \mu
\right|
\\
&=& \frac{\sqrt{2}}{\sqrt{\pi}} \frac{\sigma^3}{n \delta
  \sqrt{\delta^2 n + \sigma^2}} + \frac{\sigma^2}{n\delta^2 +
  \sigma^2} \left| \bar{x}- \mu \right|, \enas 
which yields the upper bound in \eqref{eq:21}.

Inequality (\ref{eq:21}) provides a quite concrete and intuitive idea
of the impact of the prior. First we see that, for
$n\rightarrow\infty$, the distance becomes zero, as is well known. The
prior variance $\delta^2$ has the same influence, which is also
natural given that the prior then tends towards an improper prior,
too. {{If the data are unfavourable so that $|\bar{x} - \mu|$ is large compared to $n$, then the Wasserstein distance between the two posterior distributions will be large. Due to the law of large numbers, for large $n$ the probability that $|\bar{x} - \mu| >  \delta^2 n + \sigma^2$ is small; but in contrast to such asymptotic considerations,  the bound \eqref{eq:21} makes the influence of the data on the distance explicit.}} Further the upper and lower bounds only differ
by an $O(n^{-3/2})$ term, hence at a $1/n$ precision, we have an
exact expression for the Wasserstein distance. Finally, the $O(1/n)$
term in both bounds perfectly reflects the {{intuition that}} the better
the guess of the prior mean $\mu$ (w.r.t. the data), the smaller the
influence of the prior.

\subsection{The binomial model}\label{exBayes3}

As next example we treat the case of $n$ independent and identically
distributed Bernoulli random variables with parameter of interest
$\theta\in[0,1]$; alternatively, we may say we have a single
observation $y \in \{0,1, \ldots , n\} $ from a Binomial distribution
with known $n$ and parameter of interest $\theta$. The corresponding
sampling density is
$$f (y;\theta) = {n \choose y} \theta^y (1-\theta)^{n-y}$$
and 
$p_1(\theta; y) = \kappa_1 (y) \theta^y (1-\theta)^{n-y}$ is a Beta
density with 
$$ \kappa_1(y) = \frac{1}{B(y+1, n-y+1)},$$
where $B(\cdot,\cdot)$ denotes the Beta function, 
{{and $P_1 = P_1(\cdot;y) = \mathrm{Beta}(y+1, n-y+1)$ is a Beta distribution.}}

Recall that, if 
$X \sim p(x) = \frac{1}{B(\alpha, \beta)} x^{\alpha-1}(1-x)^{\beta-1}$
then
\begin{equation*}
  \E[X] = \frac{\alpha}{\alpha +\beta}, \mbox{  }  \E[X^2] = \frac{\alpha(1+\alpha)}{(\alpha+\beta)(\alpha+\beta+1)} \mbox{ and } {\rm Var} [X] = \frac{\alpha \beta}{(\alpha+\beta)^2(\alpha+\beta+1)} .
\end{equation*}
The Stein kernel is $ \tau (x) = \frac{x(1-x)}{\alpha+\beta}$ and in
particular $\tau_1(\theta) = \frac{\theta(1-\theta)}{n+2}.$ 
Corollary \ref{varcor} gives that, for any differentiable prior 
$\pi_0$ on $\III = [0,1]$, 
\begin{equation*}\label{genbinbound}
 d_{\mathcal{W}} (P_1, P_2) \le  \sup_{0 \le \theta \le 1} | \pi_0' (\theta)| \frac{(y+1) ( n-y+1) }{(n+2)^2 (n+3)}.
\end{equation*} 
For $y$ close to $\frac{n}{2}$, this bound is of order $n^{-1}$. In
particular, for any $0 \le y \le n$, for a prior with bounded derivative, the Wasserstein distance
converges to zero as $n \rightarrow \infty$ no matter which data are
observed, but the data may affect the rate of convergence.
Next we consider some choices of prior densities which may 
 not have bounded derivatives.

 \subsubsection{Beta prior} 
For  a Beta prior 
\begin{equation}
  \label{eq:33}
  \pi_0(\theta) \propto \theta^{\alpha-1}(1-\theta)^{\beta-1},
\end{equation}
{the assumptions of Theorem \ref{maintheo} are satisfied but}
$\sup_{0 \le \theta \le 1} | \pi_0' (\theta)| $ is infinite unless
both $\alpha$ and $\beta$ are greater than or equal to 2 (or
$\alpha = \beta = 1$). Let $P_1 $ denote the ${\rm Beta}(y+1, n-y+1)$
distribution and $P_2$ the posterior distribution using the prior
\eqref{eq:33}.  It is well known that $P_2$ is again Beta distributed
: the Beta distributions are \emph{conjugate priors} for the Binomial
distribution (similarly as the normal prior is conjugate in the normal
model, see the previous section); {{in fact, it is easy to see that
    $P_2$ is the ${\rm Beta}(\alpha + y,\beta+ n-y)$ distribution.

We shall show that 
\bea \label{betabound} 
\left| \frac{y+1}{n+2}  \left(  \frac{\alpha + \beta - 2}{ n+ \alpha + \beta} \right) - \frac{\alpha - 1}{n+ \alpha + \beta} \right| &\le&   d_{\mathcal{W}} (P_1, P_2) \nonumber \\ 
  &\le & \frac{1}{n+2} \left\{ | \alpha -1| + \frac{y + \alpha}{n + \alpha + \beta} ( | \beta - 1| - | \alpha - 1|) \right\}.
\ena 

To this end, let $\Theta_1 \sim P_1$ and $\Theta_2 \sim P_2$. 
With \eqref{eq:22} we have the immediate lower bound on the Wasserstein distance, namely 
\beas
 d_{\mathcal{W}} (P_1, P_2) &  \ge & | \E[ \Theta_2] - \E[ \Theta_1] |  \\ 
&=& \left| \frac{y+1}{n+2} - \frac{y + \alpha}{n + \alpha + \beta}  \right|\\
&=& \left| \frac{y+1}{n+2}  \left( 1 - \frac{n+2}{ n+ \alpha + \beta} \right) - \frac{\alpha - 1}{n+ \alpha + \beta} \right|\\
&=& \left| \frac{y+1}{n+2}  \left(  \frac{\alpha + \beta - 2}{ n+ \alpha + \beta} \right) - \frac{\alpha - 1}{n+ \alpha + \beta} \right|  .
\enas

For an upper bound, we calculate that 
$$ \rho_0(\theta) = \frac{(\alpha -1) ( 1 - \theta) - (\beta - 1) \theta}{\theta(1-\theta)}$$ 
and hence
$$\tau_1(\theta) \rho_0(\theta) = \frac{1}{n+2} \{ (\alpha -1) ( 1 - \theta) - (\beta - 1) \theta\}.$$
Using \eqref{eq:22} we obtain the claimed upper bound 
\beas
 d_{\mathcal{W}} (P_1, P_2) &  \le &
 \frac{1}{n+2} \E \left|  (\alpha -1) ( 1 - \Theta_2) - (\beta - 1) \Theta_2 \right| \\
 &  \le &
 \frac{1}{n+2} \left\{ |  \alpha -1| \E  [ 1 - \Theta_2] + |\beta - 1| \E [ \Theta_2] \right\} \\
&=&  \frac{1}{n+2} \left\{ | \alpha -1| + \frac{y + \alpha}{n + \alpha + \beta} ( | \beta - 1| - | \alpha -1 | \right\}. 
\enas

\medskip 
Some comments on the bound \eqref{betabound} are in order. Firstly, both the upper and the lower bound vanish when $\alpha = \beta = 1$. Secondly, unless $\alpha=1$, the upper bound is  of order $O(n^{-1})$, no matter how favourable the data $y$ are.

}}

\subsubsection{The Jeffreys prior} 
An alternative popular prior is 
\begin{equation*}
  \label{eq:34}
  {{\pi}}_0(\theta) = \frac{1}{\sqrt{\theta(1-\theta)}},
\end{equation*}
the so-called Jeffreys prior obtained for $\alpha=\beta=1/2$ in
\eqref{eq:33}. This is  an improper prior {{which satisfies the assumptions of Theorem \ref{maintheo}.}}
The posterior distribution ${{P}}_2$  is 
${\rm Beta}(y + \frac12, n-y + \frac12)$. Moreover
$${{\rho}}_0(\theta) = \frac{2 \theta - 1}{ 2 {\theta(1 - \theta)}}$$
and 
$$\tau_1(\theta) {{\rho}}_0(\theta) = \frac{1}{2(n+2)} (2 \theta -1).$$

Using \eqref{eq:22}  we obtain that 
$$
\frac{1}{(n+1)} \left| \frac{y+1}{n+2} -\frac12 \right|  \le  d_{\mathcal{W}} (P_1, {{P}}_2) 
$$
and 
$$ 
d_{\mathcal{W}} (P_1, {{P}}_2)  \le  \frac{1}{n+2} \left\{  \sqrt{
\frac{\left( y + \frac12\right) \left( n - y + \frac12\right)}{(n+2)(n+1)^2}}
+ \left| \frac{y + \frac12}{n+1} -\frac12 \right|
   \right\}
$$
The upper bound follows from the Cauchy-Schwarz inequality  via
\beas 
d_{\mathcal{W}} (P_1, {{P}}_2) &\le& \frac{1}{2(n+2)}  \E | (2 \Theta_2 - 1)  | \\
&\le& \frac{1}{n+2} \left\{  \E  | \Theta_2 - \E [\Theta_2]|+ \left| \E[ \Theta_2] - \frac12 \right| \right\} \\
&\le &  \frac{1}{n+2} \left\{  \sqrt{ {\rm Var} [ \Theta_2] } + \left| \E[ \Theta_2] - \frac12 \right|   \right\} 
\\
&=& \frac{1}{n+2} \left\{  \sqrt{
\frac{\left( y + \frac12\right) \left( n - y + \frac12\right)}{(n+2)(n+1)^2}}
+ \left| \frac{y + \frac12}{n+1} -\frac12 \right|
   \right\}.
\enas 

In contrast to \eqref{betabound}, the Jeffreys prior can achieve a bound of order $O\left(n^{-\frac32}\right)$ if the data $y$ is close to $\frac{n}{2}$.

\subsection{A Poisson model}\label{sec:Bayes5}

The last case we tackle is the Poisson model with data
$x=(x_1,\ldots,x_n)$ from a Poisson distribution with sampling density
$$
f (x;\theta)=e^{-n\theta}\frac{\theta^{\sum_{i=1}^nx_i}}{\prod_{i=1}^nx_i!}.
$$
When $\sum_{i=1}^n x_i \ne 0$, which we shall now assume, then we {{obtain that $P_1$, 
the posterior distribution under a uniform prior, has pdf 
}}
\begin{equation*}
  p_1(\theta;x)\propto \exp(-\theta n)\theta^{\sum_{i=1}^nx_i+1-1}
\end{equation*}
 a
gamma density with parameters $1/n$ and $\sum_{i=1}^nx_i+1$; its Stein
kernel is simply $\tau_1(\theta)=\theta/n$ (see
Example~\ref{sec:exo4}). The general bound \eqref{eq:var} from
Corollary \ref{varcor} becomes
\begin{equation} \label{poisgen} 
 d_{\mathcal{W}} (P_1, P_2) \le  \sup_{\theta \ge 0} \left| {{\pi_0' \left(\theta;  \sum x_i  \right)}}\right| \frac{{\bar{x}}+\frac{1}{n}}{n} , 
\end{equation}
where $\bar{x} = \frac{1}{n} \sum_{i=1}^n x_i \ge \frac{1}{n}$.

Taking for $\theta$ a negative exponential prior $Exp(\lambda)$ with
$\lambda>0$, 
\begin{equation*}
  \pi_0(\theta)=\lambda e^{-\lambda\theta}
\end{equation*}
 over $\R^+$ 
yields that the posterior $P_2$ has density  
$p_2(\theta;x)\propto \exp(-\theta
(n+\lambda))\theta^{\sum_{i=1}^nx_i+1-1}$,
again a gamma density where the first parameter is updated to
$1/(n+\lambda)$. Here, the prior is monotone decreasing, hence we can
exactly calculate the effect of the prior 
to obtain
\begin{eqnarray*}
d_{\mathcal{W}}(P_1, P_2)&=&\E\left[|\log \pi_0(\Theta_2))'|\frac{\Theta_2}{n}\right]\\
&=&\lambda\frac{\E\left[\Theta_2\right]}{n}\\
&=&\lambda\frac{\bar{x}+\frac{1}{n}}{n+\lambda}\\
&=& \frac{\lambda}{n + \lambda} \bar{x} + \frac{\lambda}{n(n+\lambda)}.
\end{eqnarray*}
We note that the exact distance differs from the general bound
\eqref{poisgen} here only through a multiplicative factor
$\frac{n}{\lambda(n+\lambda)}$ (since $\sup_{\theta \ge 0} \left| {{\pi_0' \left(\theta;  \sum x_i  \right)}}\right|=\lambda^2$).  The distance increases with $\bar{x}$ but will
always be at least as large as $\frac{\lambda}{n(n+\lambda)}$. As we
assume that $\bar{x} \ge \frac{1}{n}$, the data-dependent part of the
Wasserstein distance will always be at least as large as the part
which stems solely from the prior.  Finally, from the strong law of
large numbers, $\bar{x}$ will almost surely converge to a constant as
$n \rightarrow \infty$, so that the Wasserstein distance will converge
to 0 almost surely.

\section{Technical results}\label{sec:stein-factors-1}

In this section we first prove the variant of Corollary 2.16 of
\cite{Do12} which we use in our paper. It includes
Lemma~\ref{sec:about-constants-1} as a special case.

\begin{lemma} \label{sec:technical-results}
 
 Let $P$ have a continuous density $p$  with mean $\mu$ and support
$\mathcal{I}$  an interval with
  closure $\bar{\mathcal{I}} = [a, b]$ with
  $-\infty \le a < b \le + \infty$ and let $X \sim P$. Write {{$F_P$}} for the corresponding
  cumulative distribution function. Let $h:\mathcal{I} \rightarrow \R$
  be Lebesgue-almost surely differentiable such that the Fubini
  condition
$$ \int_A \int_B | h'(v)| p(u)dvdu  =  \int_B \int_A | h'(v)| p(u)dudv < \infty  $$ 
is satisfied for all Borel-measurable subsets  $A, B \subset [a,b]$. 
Then
\begin{enumerate} 
\item \label{dobler}
\begin{equation*}
  \left| \int_a^x (h(y)-\E \left[ h({{X}}) \right]) p(y)dy   \right| 
\le   \|h'\|  \int_a^x\left(\mu -  y \right)p(y)dy;
\end{equation*}
\item for 
$g_h = \frac{\mathcal{T}_P^{-1}(h - \E \left[ h({{X}}) \right])}{\tau_P}$ it holds that 
\label{dobler:gen}  
\begin{equation*}
  \| g_h\| \le \| h'\|; 
\end{equation*}
\item \label{dobler:wasser} [Lemma~\ref{sec:about-constants-1}]
in particular, if $\mathcal{H}$ is the set of all Lipschitz-continuous
functions $h:\mathcal{I} \rightarrow \R$ with Lipschitz constant 1,  then 
$$\| g_h\| \le 1$$
for all $h \in \mathcal{H}$.
\end{enumerate} 
\end{lemma}


\begin{proof}
  We prove the three items separately, closely following \cite{Do12}
  and in particular his Lemma 5.3. 
\begin{enumerate}
\item Let $h:\mathcal{I} \rightarrow \R$ be as detailed in the
  assumptions.  Then, under the sole assumption that Fubini is
  allowed, 
  we can  write for all $a \le y \le b$ 
\begin{align*}
h(y) - \E \left[ h({{X}}) \right] & =   \int_a^b (h(y) - h(u))p(u) du \\
& = \int_a^b \int_u^yh'(v)p(u)dvdu \\
& =  \int_a^y \int_u^yh'(v) p(u) dv du - \int_y^b \int_y^uh'(v) p(u)
dv du \\
& =  \int_a^y \int_a^vh'(v) p(u) du dv - \int_y^b \int_v^bh'(v) p(u) du dv \\
& = \int_a^y F_P(v) h'(v) dv - \int_y^b (1- F_P(v)) h'(v) dv.
\end{align*}
Integrating the above w.r.t. $p$ and again applying Fubini we get
after straightforward simplifications {\color{black}{
\begin{align*}
&   \int_a^x (h(y)-\E \left[ h({{X}}) \right]) p(y)dy \\
& =
-(1-F_P(x)) \int_a^x F_P(s) h'(s)ds - F_P(x) \int_x^b (1-F_P(s)) h'(s) ds
\end{align*}}}
  for each $x \in[a, b]$ from which we readily derive 
  \begin{align*}
&    \left|   \int_a^x (h(y)-\E \left[ h({{X}}) \right]) p(y)dy \right| \\
& \le
    \|h'\| \left( (1-F_P(x)) \int_a^x F_P(s)  ds + F_P(x) \int_x^b (1-F_P(s))
      ds \right).
  \end{align*}
To deal with this last expression we use the identities 
  \begin{align*} 
    \int_a^x F_P(s) ds = x F_P(x) - \int_a^x sp(s)ds 
  \end{align*}
and 
\begin{align*}
\int_x^b
    (1-F_P(s)) ds =-x(1-F_P(x)) + \int_x^b sp(s)ds.
\end{align*}
Straightforward computations yield  the claim. 
\item For Item \ref{dobler:gen}, by definition  
\begin{equation*}
  \mathcal{T}_P^{-1}(h(x) - \E \left[ h({{X}}) \right]) = \frac{1}{p(x)} \int_a^x
  (h(y) - \E \left[ h({{X}}) \right])p(y)dy. 
\end{equation*}
Also, by definition, 
\begin{equation*}
  \tau_P(x) p(x) = \int_a^x (\mu-y) p(y)dy. 
\end{equation*}
Hence 
\begin{equation*} 
  g_h(x) = \frac{\int_a^x \left( h(y) - \E \left[ h({{X}}) \right]
    \right)p(y)dy}{\int_a^x (\mu-y) p(y)dy}
\end{equation*}
which, by Item \ref{dobler}, satisfies 
\begin{equation*}
  \| g_h\| \le \|h'\|  \left| \frac{\int_a^x (\mu-y) p(y)dy}{\int_a^x
      (\mu-y) p(y)dy} \right| = \|h'\|  .
\end{equation*}
\item Item \ref{dobler:wasser} follows directly from Rademacher's
  Theorem for Lipschitz functions which guarantees that they are
  almost surely differentiable, with derivative bounded by 1 if their
  Lipschitz constant is 1.
\end{enumerate}
\end{proof}
We conclude the paper with a proof of Proposition~\ref{prop:easycond},
restated for convenience.
\begin{proposition}  We use the notations of
  Theorem \ref{maintheo}. Suppose that $\pi_0$, $p_1$ and $p_2$ are
  differentiable over their support and that their derivatives are
  integrable. Suppose that 
  \begin{equation*}
    \lim_{x \to a_2, b_2} \pi_0(x) p_1(x) \tau_1(x)  =     \lim_{x \to a_2, b_2} p_2(x) \tau_1(x)  = 0.
  \end{equation*}
Let $\rho_1 = p_1'/p_1$ and suppose also that 
\begin{equation*}
  \pi_0' p_1 \tau_1  =  p_2'  \tau_{1}  - \rho_1 \tau_1 p_2 \in L^1(dx).
\end{equation*}
  Then  Theorem \ref{maintheo} applies.
\end{proposition}
\begin{proof} Conditions \eqref{eq:6} and \eqref{eq:9} are equivalent
  to requiring that $f_h \in \mathcal{F}_2$, in other words $(f_hp_2)$
  needs to be differentiable, $(f_hp_2)'$ needs to be integrable with
  integral on $\mathcal{I}_2$ (the support of $p_2$) equal to 0. By
  definition, 
  \begin{equation*}
    f_h(x) p_2(x) = \pi_0(x) \int_{a_1}^x (h(y) - \E[h(X_1)]) p_1(y) dy
  \end{equation*}
is differentiable if $\pi_0$ is differentiable.  Next, differentiating, 
 $$ (f_h p_2 )' (x) = \pi_0' (x) \int_{a_1}^x (h(y) - \E[h(X_1)]) p_1(y) dy + \pi_0(x) (h(x) - \E[h(X_1)]) p_1(x) .$$
 For the second summand, the Lipschitz property of $h$ gives the bound
 $$ |h(x)  - \E[h(X_1)] | \le \int_{a_1}^{b_1} |h(x) - h(y)| p_1(y) dy\le  \int_{a_1}^{b_1} |x - y| p_1(y) dy,$$ so that 
 $$ \int_{a_1}^{b_1} |  \pi_0(x) (h(x) - \E[h(X_1)])  p_1(x) | dx \le 
 \int_{a_1}^{b_1} p_2 (x)  \int_{a_1}^{b_1} |x - y| p_1(y) dy dx \le \E |X_1| + \E |X_2|,
 $$ 
  and the latter expectations are assumed to exist. Hence in order to
  guarantee \eqref{eq:6} it is sufficient to impose that 
  \begin{equation}\label{eq:11}
     \pi_0' (x) \int_{a_1}^x (h(y) - \E[h(X_1)]) p_1(y) dy \in L^1(dx).
  \end{equation}
  We can write
  \begin{align*} \int_{a_1}^x (h(y) - \E[h(X_1)]) p_1(y) dy & =p_1(x)\tau_{1}(x)g_h(x)
  \end{align*}
with 
\begin{equation*}
   g_h(x) = \frac{1}{\tau_1(x)p_1(x)}\int_{a_1}^x (h(y) - \E[h(X_1)]) p_1(y) dy
\end{equation*}
a function which we know from Lemma~\ref{sec:technical-results} to be bounded uniformly
by 1.  Hence \eqref{eq:11} (and therefore \eqref{eq:6})  boils down to a condition on 
$\pi_0'(x) p_1(x)\tau_{1}(x).$ 
Similarly \eqref{eq:9} can be tracked down to a condition on 
$\pi_0(x) p_1(x)\tau_{1}(x)$, and the claim follows. 
\end{proof}


\bibliographystyle{abbrv}

\end{document}